\documentclass[11pt,a4paper,twoside]{article}
\usepackage[UKenglish]{babel}
\usepackage[T1]{fontenc}
\usepackage[utf8x]{inputenc}
\usepackage{latexsym,amsfonts,amsmath,amsthm,amssymb, mathrsfs}
\usepackage{fullpage}
\usepackage{xcolor,bm, bbm}
\usepackage{paralist}
\usepackage{todonotes}
\usepackage{fourier}
\usepackage{amssymb}

\usepackage{cancel}

\usepackage{tikz}

\usepackage[labelformat=empty]{caption}  
\usetikzlibrary{patterns}

\DeclareMathOperator{\E}{\mathbb{E}}

\DeclareMathOperator{\R}{\mathbb R}
\DeclareMathOperator{\N}{\mathbb N}
\DeclareMathOperator{\Z}{\mathbb Z}

\makeatletter
\def\moverlay{\mathpalette\mov@rlay}
\def\mov@rlay#1#2{\leavevmode\vtop{%
   \baselineskip\z@skip \lineskiplimit-\maxdimen
   \ialign{\hfil$\m@th#1##$\hfil\cr#2\crcr}}}
\newcommand{\charfusion}[3][\mathord]{
    #1{\ifx#1\mathop\vphantom{#2}\fi
        \mathpalette\mov@rlay{#2\cr#3}
      }
    \ifx#1\mathop\expandafter\displaylimits\fi}
\makeatother

\newcommand{\cupdot}{\charfusion[\mathbin]{\cup}{\cdot}}
\newcommand{\bigcupdot}{\charfusion[\mathop]{\bigcup}{\cdot}}

\newtheorem{thm}{Theorem}[section]

\newtheorem{rem}[thm]{Remark}
\newtheorem{lem}[thm]{Lemma}
\newtheorem{thmalpha}{Theorem}

{
\theoremstyle{definition}

}

{
\theoremstyle{definition}
}

\allowdisplaybreaks
\begin{document}

\title{A sharp threshold for arithmetic effects on the tail probabilities of lacunary sums}
\author{Christoph Aistleitner, Lorenz Frühwirth, and Joscha Prochno}
\date{}

\maketitle

\begin{abstract}
A classical observation in analysis asserts that lacunary systems of dilated functions show many properties which are also typical for systems of independent random variables. For example, if $(n_k)_{k \geq 1}$ is a sequence of integers satisfying the Hadamard gap condition $n_{k+1}/n_k\geq q > 1,~k  \geq 1$, then the normalized sums $\sum_{n=1}^N \cos(2\pi n_k x)$, considered on the probability space $[0,1]$ with Borel $\sigma$-field and Lebesgue measure, satisfy the central limit theorem (CLT) and the law of the iterated logarithm (LIL). Remarkably, the situation becomes much more deliacate when the simple trigonometric function $\cos(2 \pi x)$ is replaced by a more general 1-periodic function $f$, and fine arithmetic properties of the sequence  $(n_k)_{k \geq 1}$ come into play. The most relevant arithmetic property can be phrased in terms of the number of solutions of certain 2-variable Diophantine equations. Very recently, the authors proved that the validity of the LIL requires a strictly stronger Diophantine criterion than the CLT. In the present paper we show that this is only a special case of a wide-ranging general principle: there is a sharp cutoff, which can be expressed in form of a Diophantine criterion on the sequence $(n_k)_{k \geq 1}$, at which the tail probabilities of normalized sums $\sum_{k=1}^N f(n_k x)$ change from Gaussian to potentially erratic behavior. More precisely, let $L(N,a,b,c)$ be the number of solutions $(k,\ell)$ of the equation $a n_k - b n_\ell = c$, where $1\leq k,\ell \leq N$. Roughly speaking, we prove: if $L(N,a,b,c) \leq N / g_N$ for some $g_N$, then $\mathbb{P} \left[\sum_{k=1}^N f(n_k x)  > t \|f\|_2 \sqrt{N} \right]$ is asymptotically is accordance with standard normal behavior for all $t$ up to $\sqrt{2 \log g_N}$. We also show that this criterion is optimal in the sense that under the same premises, the conclusion can fail to be true for values of $t$ beyond this threshold. 
\end{abstract}

\section{Introduction and main results}

A sequence of positive integers $(n_k)_{k \geq 1}$ is said to satisfy the Hadamard gap condition if there is a number $q>1$ such that
    \begin{align}\label{eq:HGap}
        \frac{n_{k+1}}{n_k} \geq q, \qquad k \geq 1.
    \end{align}
Sequences satisfying this gap condition are also called lacunary sequences, and the trigonometric systems $S_N(\cdot)=\sum_{k=1}^N\cos(2\pi n_k \cdot)$ of functions on $[0,1]$ associated with a lacunary sequence are referred to as lacunary trigonometric systems. When $[0,1]$ is endowed with the Borel $\sigma$-field and Lebesgue measure $\lambda$, the functions $S_N$ can be considered as partial sums of random variables $X_k=\cos(2\pi n_k \cdot)$, $k\in\N$, which are identically distributed and, if all $n_k$, $k\in\N$, are distinct,  also uncorrelated. However, the elements of the sequence $(X_k)_{k\in\N}$ are not independent. Still, as the last decades have shown, as a consequence of the Hadamard gap condition, such systems $S_N$, $N\in\N$, asymptotically behave like sums of independent random variables, and as such satisfy a number of classical limit theorems of probability theory. For instance, in 1939 Kac proved a central limit theorem for a sequence $n_k$, $k\in\N$, with large gaps \cite{Kac_biggaps}, that is, if $n_{k+1}/n_k\to\infty$ as $k\to\infty$. In 1947, Salem and Zygmund proved in \cite{SZ1947} that a central limit theorem holds under the Hadamard gap condition \eqref{eq:HGap}, showing that for any $t\in\R$,
    \begin{align*}
        \lim_{N\to\infty} \lambda\Big( \Big\{\omega\in[0,1]\,:\, \sum_{k=1}^N X_k(\omega) \leq t\sqrt{n/2} \Big\} \Big) = \frac{1}{\sqrt{2\pi}}\int_{-\infty}^t e^{-y^2/2}\,dy.
    \end{align*} 
Since the variance of $\cos(2\pi n_k \cdot)$ is $1/2$, and so 
  \begin{equation} \label{asym_var}
    \int_0^1 \Big(\sum_{k=1}^N X_k(\omega) \Big)^2\,d\omega = \frac{N}{2},
  \end{equation}
the result captures what is expected in the truly independent setting; this result has been significantly strengthened since then, and we refer, e.g., to the work of Berkes \cite{ba} and of Philipp and Stout \cite{psa}.    
Almost a decade later, Salem and Zygmund \cite{SZ1950} and Erd\"os and G\'al \cite{EG1955} studied the behavior beyond the scale of normal fluctuations, obtaining that under the Hadamard gap condition \eqref{eq:HGap} a law of the iterated logarithm holds, i.e., for $\lambda$-almost all $\omega\in[0,1]$,
    \begin{align*}
        \limsup_{N\to\infty} \frac{\sum_{k=1}^N X_k(\omega)}{\sqrt{N \log \log N}} = 1.
    \end{align*}
Note that again this is in perfect accordance with independent behavior, with the ``correct'' normalization term coming from \eqref{asym_var}. It was soon understood that such results do not transfer to the general setting of $1$-periodic, centered, and sufficiently regular functions; in fact, they do not even remain valid for trigonometric polynomials, which arguably are the simplest periodic functions in this context. More specifically, an example attributed to Erdös and Fortet considers the function
  \[
    f(x) := \cos(2\pi x) + \cos(4\pi x)
  \]
and the gap sequence $n_k:=2^k-1$, $k\in\mathbb N$, for which it can be shown that the sequence of partial sums $N^{-1/2}\sum_{k=1}^N f(n_k \cdot)$, $N\in\mathbb N$, has a non-Gaussian limit distribution commonly referred to as a ``variance mixture Gaussian''. For a detailed explanation, we refer to \cite{ABT2024}, but at its core, the failure of a Gaussian limit theorem stems from the fact that the linear Diophantine equation 
  \[
    n_{k+1} - 2n_k = 1
  \]
possesses too many solutions in $k$ for the sequence $(n_k)_{k\geq 1}$ (for more information on the Erd\H os--Fortet example, see also \cite{cb,fm}). This already shows that there is a delicate interplay between arithmetic and analytic properties that affect the probabilistic behavior. A characterization under which central limit behavior as in the truly independent case can be observed was obtained by Aistleitner and Berkes. We consider the linear two-variable Diophantine equation 
  \[
    a n_k - b n_\ell = c
  \]
with $a,b\in\mathbb N$ and $c\in\mathbb Z_{\geq 0}$ fixed, and are interested in the number of solutions $(k,\ell)$, when the size of $k$ and $\ell$ is bounded above by some threshold value. More precisely, for $a,b \in \N$, let
\begin{equation}
    \label{Eq_numb_sol}
L(N,a,b) := \sup_{c \in \Z} \# \big\{1\leq k, \ell \le N \,:\, a n_k - bn_\ell =c \big\},
\end{equation}
where for $a=b$, the supremum is only taken over $c \neq 0$; note that, due to symmetry, one can always take the supremum over all non-negative integers $c$, and we allow the implicit constant to depend on $a,b$. In 2010, Aistleitner and Berkes showed in \cite{AB2010} the remarkable result that assuming $L(a,b,N) = o(N)$ is sufficient for 
\[
\frac{S_N(x)}{\sqrt{N}} := \frac{1}{\sqrt{N}}\sum_{k=1}^N f(n_k x), \quad x\in[0,1],
\]
to satisfy a central limit theorem with asymptotic variance $\| f \|_2^2 := \int_0^1 f(x)^2 dx$, where $f$ can be taken from a reasonably large class of periodic functions. In \cite{AB2010}, the authors also show that this Diophantine condition is optimal. Shortly after, Aistleitner was able to establish laws of the iterated logarithm (LIL) for similar types of lacunary sums \cite{A2010}. However, there is a notable drawback; in order to establish LILs that are in accordance with truly independent behavior, it was necessary to assume $L(N, a,b) = O\left( \frac{N}{(\log N)^{1+ \varepsilon}} \right)$, for some $\varepsilon > 0$. For a decade it was believed that the additional logarithmic factor in this condition (in comparison with the criterion for the CLT case) was an artifact stemming from the method of proof, but as was shown by the authors of the present paper in \cite{AFP_Lil}, it is actually necessary in general. It thus seems that there is an intricate interplay between the number of Diophantine solutions $L(N, a,b)$ on the one hand, and the question to what extent the tail probabilities of $S_N$ behave like those of a sum of i.i.d.~random variables on the other hand. In this paper, we establish such a connection to be true as a general principle: having better control over the number of solutions of the relevant Diophantine equations allows one to have control of the tail probabilities of the lacunary sums further out in the tail. More precisely, we show that if $L(N,a,b)= O\left(\frac{N}{g_N} \right)$, for some sufficiently well-behaved sequence $(g_N)_{N \in \N}$ with $1 \leq g_N \leq N$, then for any mean-zero trigonometric polynomial $f$, we have
\[
\frac{\mathbb{P} \left[ S_N(x) \geq t \| f \|_2 \sqrt{2N \log g_N} \right]}{1-\Phi(t \sqrt{2 \log g_N})} \stackrel{N \rightarrow \infty}{\longrightarrow} 1
\]
as long as $t \in (0,1)$, where $\Phi$ denotes the standard normal distribution function. If on the other hand $t > 1$, then we can find a trigonometric polynomial $f$ such that
\[
\frac{\mathbb{P} \left[ S_N(x) \geq t \| f \|_2 \sqrt{2N \log g_N} \right]}{1-\Phi(t \sqrt{2 \log g_N })} \stackrel{N \rightarrow \infty}{\longrightarrow} \infty.
\]
This shows that when imposing a bound $L(N,a,b)= O\left(\frac{N}{g_N} \right)$ on the maximal number of solutions of the Diophantine equations, the tail behavior of $S_N$ undergoes a transition at the critical value $\sqrt{2 \log g_N }$: from a range where the tail probabilities follow the standard normal distribution as a consequence of the bound for the number of solutions of the Diophantine equations, into a range where the Diophantine assumption loses its usefulness and the tail probabilities can start to deviate strongly from those of the standard normal distribution. Note that our theorems are consistent with the threshold identified in \cite{A2010, AFP_Lil} for the particular case of the LIL, where the relevant values of $g_N$ were $g_N = (\log N)^{1+\varepsilon}$ resp.\ $g_N = (\log N)^{1-\varepsilon}$ . Interestingly, the whole phenomenon discussed in the present paper does not exist if one considers only pure cosine lacunary sums $\sum_{k =1}^N \cos(2 \pi n_k x)$; see also \cite{AFHM_mgf} for this case. \\

Our first theorem contains the sufficiency part of the general principle described above: the better control one has over arithmetic effects, the further out one can control the behavior of the tail probabilities. 
\begin{thmalpha}
\label{thm_iid}
Let $(g_N)_{N \in \N}$ be a sequence such that $1 \leq g_N \leq N$ for all $N\in\N$ and so that both $g_N$ and $\frac{N}{g_N}$ are monotone increasing in $N$ with $g_N \longrightarrow \infty$. Let $f: \mathbb{R} \to \mathbb{R}$ be a non-vanishing trigonometric polynomial with mean zero and degree $d$.  Assume that
\begin{equation}
\label{Eq_Cond_Dio_1}
 L(N,a,b)= O\left( \frac{N}{g_N} \right) \qquad \text{for all $a,b \in \{1, \dots, d\}$.}
\end{equation}
Let $\varepsilon > 0$. Then 
\[
\lim_{N \rightarrow \infty} \,\sup_{0 \le t \le  (1-\varepsilon) \sqrt{2 \log g_N}} \frac{\mathbb{P} \left[S_N \geq t \|f\|_2  \sqrt{N}  \right]}{1 - \Phi(t)} = 1.
\]
\end{thmalpha}

As a corollary of Theorem \ref{thm_iid}, we recapture the main result of \cite{AB2010}, namely the central limit theorem for $\sum f(n_k x)$ under the Diophantine condition $L(N,a,b) = o(N)$. Indeed, applying Theorem \ref{thm_iid} with any divergent $g_N$ allows us to conclude that $\mathbb{P} \left[S_N \leq t \| f\|_2  \sqrt{N}  \right] \to \Phi(t)$ for any fixed $t\in\mathbb R$. Theorem \ref{thm_iid} also contains the key ingredient for the proof of the main result of \cite{A2010}, namely for the law of the iterated logarithm for $\sum f(n_k x)$ under the Diophantine condition $L(N,a,b) = O(N / (\log N)^{1 + \eta})$ for some $\eta > 0$. By Theorem \ref{thm_iid}, the Diophantine assumption $L(N,a,b) = O(N / (\log N)^{1 + \eta})$ allows us to control the tail probabilities for $t$ in the range up to $t \leq (1-\varepsilon) \sqrt{1 + \eta} \sqrt{2 \log \log N}$, which leads to the LIL when $\varepsilon>0$ is chosen so small in comparison with $\eta$ that $(1-\varepsilon)\sqrt{1+\eta} > 1$.\\

Our second theorem contains the necessity part of the general principle mentioned before, by showing that the Diophantine criterion in Theorem \ref{thm_iid} is essentially optimal: a slightly weaker Diophantine assumption would not allow to control the tail probabilities throughout such a large range. 
\begin{thmalpha}
\label{thm_not_iid}
Let $(g_N)_{N \in \N}$ be a sequence such that $g_N$ and $\frac{N}{g_N}$ are both monotonically increasing and diverging (as $N \to \infty$). Then,  for every given $\varepsilon > 0$, there exist a lacunary sequence $(n_k)_{k \in \N}$ and a mean zero trigonometric polynomial $f$ such that for infinitely many values of $N$ we have
\begin{equation}
\label{Eq_Cond_Dio}
 L(N,a,b)= O\left( \frac{N}{g_{N}} \right)  
\end{equation}
but
\[
\limsup_{N \rightarrow \infty} \frac{ \mathbb{P} \left[ S_{N} \geq t \sqrt{N}  \right]}{ 1- \Phi(t )} = \infty,
\]
where $t = (1+ \varepsilon) \| f \|_2  \sqrt{ 2 \log g_{N}} $.
\end{thmalpha}

Broadly speaking, the results in this paper show how upper bounds on $L(N,a,b)$ cause the tail probabilities of the lacunary sum to be (asymptotically) in accordance with those of the standard normal distribution. Note, however, that the assumption $L(N,a,b) = O(N/g_N)$ only makes sense when $g_N$ is smaller than $N$, since $L(N,a,b)$ by definition always is a positive integer. Thus the relation between two-term Diophantine equations and tail probabilities of lacunary sums cannot extend any further than out to tails of order $\sqrt{2 \log N}$ (where the tail probabilities of the standard normal distribution are decaying as $\approx 1/N$). Accordingly, the tail probabilities of lacunary sums at the scale of larger deviations must be controlled by different effects, other than two-term Diophantine equations; this is a very interesting topic which was touched upon in \cite{agkpr,fjp}, but which is out of scope of the present paper. 

\section{Proof of Theorem \ref{thm_iid}}

Our proof uses martingale methods, which were introduced to the theory of lacunary sums of dilated functions independently by Berkes \cite{ba} and Philipp and Stout \cite{psa} (see also Berkes \cite{berk2,berk1}). We will apply the following result of Grama \cite[Theorem 2.5]{grama} (with the choice of $\delta=1$ in his more general formulation).

\begin{thm}
Let $X^n = (X_k^n,\mathcal{F}_k^n)_{0 \leq k \leq n}$ be a triangular array of square integrable martingales, where $X_0^n=0$ almost surely. Let $\xi_k^n = X_k^n - X_{k-1}^n$ and 
$$
\langle X^n \rangle_k = \sum_{0 < i \leq k} \E \left( \left(\xi_i^n\right)^2 | \mathcal{F}_{i-1}^n \right), \qquad k=1,\dots,n, \quad n =1,2,\dots.
$$
Let 
$$
L_4^n = \E \left( \sum_{0 < i \leq n} \left| \xi_i^n \right|^4 \right) 
$$
and 
$$
N_4^n = \E \left( \left| \langle X^n \rangle_n -1 \right|^2 \right).
$$
Assume that $L_4^n + N_4^n < 1$ and $0 <\eta < 1$. Then uniformly in $r$, subject to $0 \leq r \leq \sqrt{2 \eta \left| \log \left (L_4^n + N_4^n \right)\right|}$, we have
$$
\frac{\mathbb{P} (X_n^n \geq r)}{1 - \Phi(r)} \to 1, \qquad \frac{\mathbb{P} (X_n^n \leq -r)}{\Phi(-r)} \to 1, \qquad \text{as $n \to \infty$.}
$$
\end{thm}

For simplicity of writing we assume throughout the proof that $f$ is an even function, i.e.\ $f$ is a sum of cosine-terms. The proof in the general case (functions $f$ which include also sine-terms) is exactly the same, only that the formulas become slightly more complicated. Accordingly, assume that
$$
f(x) = \sum_{j=1}^d c_j \cos(2 \pi j x)
$$
(note that there is no term for $j=0$ since by assumption $f$ has mean zero). \\

Throughout this section, implied constants are only allowed to depend on the sequence $(n_k)_{k \geq 1}$ and its growth factor $q$, as well as on the function $f$ and on the implied constants in assumption \eqref{Eq_Cond_Dio_1}.\\

Let $N\in\mathbb N$ be given, and assume that $N$ is ``large''. We decompose the index set $\{1, \dots, N \}$ into a disjoint union 
\begin{equation} \label{eq_disjoint_union}
\{1 , \dots, N\} = \Delta_1\cupdot \Delta_1' \cupdot \Delta_2 \cupdot \Delta_2' \cupdot \dots \cupdot \Delta_n \cupdot \Delta_n',
\end{equation}
such that all elements of $\Delta_1$ are smaller than all elements of $\Delta_1'$, all elements of $\Delta_1'$ are smaller than all elements of $\Delta_2$, and so on. We do this in such a way that, for all $i\in\{1,\dots,n\}$,
$$
\# \Delta_i = (\log N)^5 \quad\text{and} \quad \# \Delta_i' = 6 \log_q N, 
$$
where $q$ is the growth factor of the lacunary sequence. For ease of writing, we will treat $(\log N)^5$ and $6 \log_q N$ as integers, and assume that $N$ is indeed divisible without remainder by their sum -- otherwise one would have to use Gauss brackets to make everything integer-valued (at the expense of formulas which are much more difficult to read), which would introduce small error terms that do not affect the overall result. By construction, we then have
\begin{equation} \label{eq_number_blocks}
n = O \left(N / (\log N)^5 \right),
\end{equation}
and
\begin{equation} \label{eq_**}
N - \sum_{i=1}^n \# \Delta_i = \sum_{i=1}^n \#\Delta_i' = O \left(N / (\log N)^4 \right).
\end{equation}
The purpose of the ``dashed'' blocks $\Delta_i'$ is to separate the elements of the ``un-dashed'' blocks $\Delta_i$ for different values of $i\in\{1,\dots,n\}$. The overall contribution coming from indices in the dashed blocks will turn out to be negligible in comparison to the contribution of the indices in the un-dashed blocks. \\

For $1 \leq i \leq n$, we let $\mathcal{F}_i$ denote the $\sigma$-field generated by all intervals of the form 
$$
\left[ j 2^{-\left\lceil 3 \log_2 N + \log_2 (\max_{k \in \Delta_{i}} n_k ) \right\rceil}, (j+1) 2^{-\left\lceil 3 \log_2 N + \log_2 (\max_{k \in \Delta_{i}} n_k ) \right\rceil} \right), \qquad 0 \leq j < 2^{\left\lceil 3 \log_2 N + \log_2 (\max_{k \in \Delta_{i}} n_k ) \right\rceil}.
$$
For the sake of completeness, we also define $\mathcal{F}_0$ as the trivial sigma-field consisting only of $\emptyset $ and  $[0,1]$. The dyadic construction ensures that $\mathcal{F}_1, \dots, \mathcal{F}_n$ is a filtration, and will also ensure that for $k \in \Delta_i$ we have $\E(f(n_k x) | \mathcal{F}_{i}) \approx f(n_k x)$ for all $x$, while $\E(f(n_k x) | \mathcal{F}_{i-1}) \approx 0$. This is a consequence of the quick growth of the lacunary sequence, and of the presence of the ``dashed'' blocks in \eqref{eq_disjoint_union}. More precisely, let
$$
\xi_i = \frac{1}{\|f\|_2 \sqrt{\sum_{i=1}^n \# \Delta_i}} \left( \E \left( \sum_{k \in \Delta_i} f(n_k x) \Big| \mathcal{F}_{i} \right) - \E \left( \sum_{k \in \Delta_i} f(n_k x) \Big| \mathcal{F}_{i-1} \right) \right).
$$
Then, by construction, $\xi_i, ~1 \leq i \leq n$, is a martingale difference, and $X_i = \sum_{j=1}^i \xi_i,~1 \leq j \leq n$, is a martingale (which of course is square-integrable, since everything is finite), which is adapted to the filtration $(\mathcal{F}_i)_{1 \leq i \leq n}$. \\

A simple calculation shows that for any positive real numbers $a,b,c$ with $a<b$, and for any 1-periodic function $g$ with mean zero, one has
$$
\left| \int_a^b g(c x ) ~dx \right| \leq \frac{1}{c} \|g\|_\infty
$$
(see \cite[Lemma 2.2]{AB2010} for a proof). Thus, for any $k \in \Delta_i$, we have
\begin{eqnarray}
\E \left( \sum_{k \in \Delta_i} f(n_k x) \Big| \mathcal{F}_{i-1} \right) & \leq & 2^{\left\lceil 3 \log_2 N + \log_2 (\max_{k \in \Delta_{i-1}} n_k ) \right\rceil} \sum_{k \in \Delta_i} \frac{1}{n_k} \|f\|_\infty \nonumber\\
& = & O \left( \frac{N^3 \# \Delta_i \max_{k \in \Delta_{i-1}} n_k}{\min_{k \in \Delta_{i}} n_k } \right)  \nonumber\\
& = & O \left(N^3 (\log N)^5 q^{-6 \log_q N} \right) \nonumber\\
& = & O \left(N^{-3} (\log N)^5 \right), \label{eq_small_int}
\end{eqnarray}
as $N \to \infty$. Note how this argument used the quick growth of the sequence $(n_k)_{k \geq 1}$, and the fact that by construction the indices in $\Delta_{i-1}$ and those in $\Delta_i$ are well-separated from each other.\\

Taking derivatives, we note that 
$$
\left| \frac{d}{dx} f(n_k x) \right| = O (n_k), \qquad \text{uniformly in $x$}. 
$$
Thus, in view of the length of the intervals that generate $\mathcal{F}_{i}$, we have
\begin{eqnarray*}
 \left| \sum_{k \in \Delta_i} f(n_k x)  - \E \left( \sum_{k \in \Delta_i} f(n_k x) \Big| \mathcal{F}_{i} \right) \right| & = & O \left(\sum_{k \in \Delta_i}  n_k 2^{-\left\lceil 3 \log_2 N + \log_2 (\max_{k \in \Delta_{i}} n_k ) \right\rceil} \right) \\
 & = & O \left( \sum_{k \in \Delta_i} N^{-3} \right) \\
 & = & O \left( N^{-3} (\log N)^5 \right),
\end{eqnarray*}
uniformly in $x$. Overall, in view of \eqref{eq_**} (where it is shown that $ \sum_{i=1}^n \# \Delta_i = \Theta(N)$), this yields
\begin{eqnarray}  \label{eq_diff_xi}
\left|\xi_i - \frac{1}{\|f\|_2 \sqrt{\sum_{i=1}^n \#\Delta_i}} \sum_{k \in \Delta_i} f(n_k x) \right| = O \left(N^{-7/2} (\log N)^5 \right),
\end{eqnarray}
and
\begin{equation} 
\label{eq_in_part}
\begin{split}
 \left| \xi_i \right| & \le \left| \xi_i - \frac{1}{\|f\|_2 \sqrt{\sum_{i=1}^n \#\Delta_i}} \sum_{k \in \Delta_i} f(n_k x)  \right |  + \left|  \frac{1}{\|f\|_2 \sqrt{\sum_{i=1}^n \#\Delta_i}} \sum_{k \in \Delta_i} f(n_k x) \right |\\
 &= O \left( N^{-7/2} (\log N )^5 \right) + O \left( N^{-1/2} (\log N)^5 \right) \\
 & = O \left( N^{-1/2} (\log N)^5 \right).
\end{split}
\end{equation}
By \eqref{eq_number_blocks} and \eqref{eq_diff_xi}, we have 
\begin{equation} 
\label{eq_x_dist}
\left| X_n -  \frac{1}{\|f\|_2 \sqrt{\sum_{i=1}^n \#\Delta_i}} \sum_{i=1}^n \sum_{k \in \Delta_i} f(n_k x) \right| = O \left(N^{-5/2} \right).
\end{equation}\\

Thus, in the notation of Theorem A, using \eqref{eq_number_blocks}, \eqref{eq_**} and \eqref{eq_in_part}, we obtain
\begin{equation}
L_4^n = \E \left( \sum_{i=1}^n |\xi_i|^4 \right) = O \left( \sum_{i=1}^n  \frac{(\log N)^{20}}{N^2} \right) = O\left( \frac{(\log N)^{15}}{N} \right). \label{eq_L_4_est}
\end{equation}
The next (and more complicated) step is to estimate $N_4^n$. Using the inequality $|x^2-y^2| \leq |x-y| \cdot |x+y|$ as well as  \eqref{eq_diff_xi} and \eqref{eq_in_part}, we have
\begin{eqnarray*}
& & \left| \xi_i^2 - \left( \frac{1}{\|f\|_2 \sqrt{\sum_{i=1}^n \# \Delta_i}} \sum_{k \in \Delta_i} f(n_k x) \right)^2 \right|\\ 
& \leq & \left| \xi_i - \left( \frac{1}{\|f\|_2 \sqrt{\sum_{i=1}^n \# \Delta_i}}  \sum_{k \in \Delta_i} f(n_k x) \right) \right| \cdot \left| \xi_i + \left( \frac{1}{\|f\|_2 \sqrt{\sum_{i=1}^n \# \Delta_i}}  \sum_{k \in \Delta_i} f(n_k x) \right) \right| \\
& = & O  \left( N^{-4} (\log N)^{10} \right).
\end{eqnarray*}
Thus, by Minkowski's inequality,
\begin{eqnarray}
\sqrt{N_4^n} &  = & \left( \E \left( \sum_{i=1}^n \E \left(\xi_i^2 | \mathcal{F}_{i-1} \right) -1 \right)^2 \right)^{1/2} \nonumber\\ 
& = & \left( \E \left( \sum_{i=1}^n \E \left( \frac{1}{\|f\|_2^2 \sum_{i=1}^n \# \Delta_i} \left( \sum_{k \in \Delta_i} f(n_k x) \right)^2 \Big| \mathcal{F}_{i-1} \right) -1 \right)^2 \right)^{1/2} + O  \left( N^{-3} (\log N)^{15} \right). \label{eq_mink}
\end{eqnarray}
Now we note that
\begin{eqnarray*}
& & \left( \sum_{k \in \Delta_i} f(n_k x) \right)^2 \\
& = &  \left( \sum_{k \in \Delta_i} \sum_{j=1}^d c_j \cos(2 \pi j n_k x) \right)^2 \\
& = & \sum_{k,\ell \in \Delta_i} \sum_{j_1,j_2=1}^d c_{j_1} c_{j_2} \cos(2 \pi j_1 n_{k} x) \cos(2 \pi j_2 n_{\ell} x)  \\
& = &  \sum_{k,\ell \in \Delta_i} \sum_{j_1,j_2=1}^d \frac{c_{j_1} c_{j_2}}{2}  \left( \cos(2 \pi (j_1 n_{k} + j_2 n_{\ell}) x) + \cos(2 \pi (j_1 n_{k} - j_2 n_{\ell}) x) \right) \\
 & = &  \sum_{h \in \mathbb{Z}} \sum_{k,\ell \in \Delta_i} \sum_{j_1,j_2=1}^d \frac{c_{j_1} c_{j_2}}{2}  \left(\mathbb{1}(j_1 n_{k} + j_2 n_{\ell} = h) + \mathbb{1}(j_1 n_{k} - j_2 n_{\ell} = h)   \right) \cos(2 \pi h x) \\
& = & \#\Delta_i \|f\|_2^2 +  \sum_{h \in \mathbb{Z}} \underbrace{\sum_{k,\ell \in \Delta_i} \sum_{j_1,j_2=1}^d}_{(j_1,k) \neq (j_2,\ell)} \frac{c_{j_1} c_{j_2}}{2}  \left(\mathbb{1}(j_1 n_{k} + j_2 n_{\ell} = h) + \mathbb{1}(j_1 n_{k} - j_2 n_{\ell} = h)   \right) \cos(2 \pi h x),
\end{eqnarray*}
which shows how the two-term Diophantine equations come into play. If $|h| \geq \min_{m \in \Delta_i} n_m$, then, similar to \eqref{eq_small_int}, we have
\begin{equation} \label{int_small}
\big| \E \big( \cos(2 \pi h x) | \mathcal{F}_{i-1} \big) \big| = O \left(N^{-3} (\log N)^5 \right).
\end{equation}
Clearly, when $k,\ell \in \Delta_i$, then the sum $j_1 n_{k} + j_2 n_{\ell}$ always has a value exceeding $\min_{m \in \Delta_i} n_m$, but for $j_1 n_{k} - j_2 n_{\ell}$ there may be some configurations of $j_1,j_2,k,\ell$ for which $\left|j_1 n_{k} - j_2 n_{\ell}\right| \leq \min_{m \in \Delta_i} n_m$. 
Using \eqref{int_small} we have
\begin{align*}
& \left| \sum_{\substack{h \in \mathbb{Z} \\ |h| \ge \min_{m \in \Delta_i} n_m  }} \underbrace{\sum_{k,\ell \in \Delta_i} \sum_{j_1,j_2=1}^d}_{(j_1,k) \neq (j_2,\ell)} \frac{c_{j_1} c_{j_2}}{2}  \left(\mathbb{1}(j_1 n_{k} + j_2 n_{\ell} = h) + \mathbb{1}(j_1 n_{k} - j_2 n_{\ell} = h)   \right) \mathbb{E} \left[ \cos(2 \pi h x) | \mathcal{F}_{i-1} \right] \right| \\
& = O( \Delta_i^2) \times O(N^{-3} (\log N)^5 ) \\
& = O(N^{-3} (\log N)^{15}).
\end{align*}
Thus, 
\begin{eqnarray*}
 & & \E \left( \frac{1}{\|f\|_2^2} \left( \sum_{k \in \Delta_i} f(n_k x) \right)^2 \Big| \mathcal{F}_{i-1} \right) - \# \Delta_i \\
& = &  \E \left( \frac{1}{\|f\|_2^2} \sum_{\substack{h \in \mathbb{Z},\\ |h| \leq \min_{m \in \Delta_i} n_m}} \underbrace{\sum_{k,\ell \in \Delta_i} \sum_{j_1,j_2=1}^d}_{(j_1,k) \neq (j_2,\ell)} \frac{c_{j_1} c_{j_2}}{2} \mathbb{1}   (j_1 n_{k} - j_2 n_{\ell} = h)  \cos(2 \pi h x) \Bigg| \mathcal{F}_{i-1} \right) + O(N^{-3} (\log N)^{15}),
\end{eqnarray*}
which together with \eqref{eq_mink} yields
\begin{eqnarray*}
& & \sqrt{N_4^n} \\
&  =  &  \frac{1}{\|f\|_2^2 \sum_{i=1}^n \# \Delta_i}  \left(\E \left( \sum_{i=1}^n \E \left( \sum_{\substack{h \in \mathbb{Z},\\ |h| \leq \min_{m \in \Delta_i} n_m}}  \underbrace{\sum_{k,\ell \in \Delta_i} \sum_{j_1,j_2=1}^d}_{(j_1,k) \neq (j_2,\ell)}   \frac{c_{j_1} c_{j_2}}{2}   \mathbb{1}   (j_1 n_{k} - j_2 n_{\ell} = h)  \cos(2 \pi h x) \Bigg| \mathcal{F}_{i-1} \right) \right)^2 \right)^{1/2} \\
& & +  O(N^{-3} (\log N)^{15}).
\end{eqnarray*}

In the following, we will give an upper bound for
\begin{eqnarray}
& & \E \left( \sum_{i=1}^n \E \left( \sum_{\substack{h \in \mathbb{Z},\\ |h| \leq \min_{m \in \Delta_i} n_m}} \underbrace{\sum_{k,\ell \in \Delta_i} \sum_{j_1,j_2=1}^d}_{(j_1,k) \neq (j_2,\ell)}  \frac{c_{j_1} c_{j_2}}{2}   \mathbb{1}   (j_1 n_{k} - j_2 n_{\ell} = h)  \cos(2 \pi h x) \Bigg| \mathcal{F}_{i-1}  \right) \right)^2 \label{eq_jensen}.
\end{eqnarray}
For $1 \le i \le n$ and for $h \in \Z$, we define
\[
r_{h,i} := \mathbb{1}(|h| \le \min_{m \in \Delta_i} n_m)\underbrace{\sum_{k,\ell \in \Delta_i} \sum_{j_1,j_2=1}^d}_{(j_1,k) \neq (j_2,\ell)} \frac{c_{j_1} c_{j_2}}{2}   \mathbb{1}   (j_1 n_{k} - j_2 n_{\ell} = h)
\]
as well as
$$
r_h  =  \sum_{\substack{1 \leq i \leq n}} r_{h,i}.
$$
so that the expression in line \eqref{eq_jensen} becomes
$$
\E \left( \sum_{h,k \in \Z} \sum_{i,j=1}^n r_{h,i} r_{k,j} \mathbb{E}( \cos(2 \pi h) | \mathcal{F}_{i-1}) \E ( \cos(2 \pi k) | \mathcal{F}_{j-1}) \right).
$$
By the construction of $\mathcal{F}_{i-1}$, for $|h| \le \min_{m \in \Delta_i} n_m$, we have $|\cos(2 \pi h x) - \E ( \cos(2 \pi h x) | \mathcal{F}_{i-1}) | = O\left(N^{-3} \right)$, uniformly in $x$, and thus
\begin{align*}
& \E \left( \sum_{h,k \in \Z} \sum_{i,j=1}^n r_{h,i} r_{k,j} \mathbb{E}( \cos(2 \pi h x) | \mathcal{F}_{i-1}) \E ( \cos(2 \pi k x ) | \mathcal{F}_{j-1}) \right) \\
& \qquad  = \E \left( \sum_{h,k \in \Z} \sum_{i,j=1}^n r_{h,i} r_{k,j}  \cos(2 \pi h x)  \cos(2 \pi kx )  \right) + O(N^{-3}) \sum_{h,k \in \Z} \sum_{i,j=1}^n |r_{h,i}| |r_{k,j}| \\
& \qquad = r_0^2 + \sum_{h \neq 0} \frac{r_h^2}{2} + O(N^{-3}) \left( \sum_{h \in \Z} \sum_{i=1}^n |r_{h,i}| \right)^2.
\end{align*}
Note that for any fixed $j_1, j_2$, and for any $i$ and any fixed $k \in \Delta_i$, there can be at most two different values of $\ell \in \Delta_i$ for which $|j_1 n_{k} - j_2 n_{\ell}| \leq \min_{m \in \Delta_i} n_m$. This means that for the numbers $r_{h,i}$ defined above we have
$$
\sum_{h \in \mathbb{Z}} \sum_{i=1}^n |r_{h,i}| = O(N),
$$
which in particular implies that $\sum_{h \in \Z} |r_h| = O(N)$. The above estimate leaves us with 
\[
O \left(N^{-3} \right) \left( \sum_{h \in \Z} \sum_{i=1}^n |r_{h,i}| \right)^2 = O\left(N^{-1} \right).
\]

On  the other hand, by assumption \eqref{Eq_Cond_Dio_1} we know that for any $h \in \Z$ and any fixed $j_1, j_2 \leq d$ the number of solutions of $j_1 n_k - j_2 n_\ell = h$, subject to $k,\ell \le N$, is at most $O(N/g_N)$, uniformly in $h$ and thus
\begin{equation} \label{four-upper}
\sup_{h \in \Z}|r_{h}| = O(N/g_N).
\end{equation}
Consequently, the expression in line \eqref{eq_jensen} is bounded above by
\begin{equation} \label{four-upper-2}
 r_0^2 + \sum_{h \neq 0} \frac{r_h^2}{2} + O \left(N^{-1} \right) = O \left( \max_{h \in \mathbb{Z}} |r_h| \sum_{h \in \mathbb{Z}} |r_h| \right)+ O \left(N^{-1} \right)  = O \left(N^2 / g_N \right),
\end{equation}
where we used that $g_N \leq N$ by assumption. Combining all estimates, we arrive at
$$
\sqrt{N_4^n} = O \left(1 / \sqrt{g_N} \right) + O\left(N^{-3} (\log N)^{15} \right), 
$$
and thus, again using that $g_N \leq N$, by assumption, 
\begin{equation} \label{eq_N_4_est}
N_4^n = O (1/g_N).
\end{equation}
Note how \eqref{four-upper} and \eqref{four-upper-2} were derived using the Diophantine assumption from the statement of the theorem (and actually this is the only place where this assumption is used during the proof of Theorem \ref{thm_iid}). The role of the Diophantine assumption is essentially to ensure that the centered fourth moment of the lacunary sum is not too large, which is reflected in the upper bound \eqref{eq_N_4_est} for the centered fourth moment for the martingale.\\

We can now control the sizes of $L_4^n$ and $N_4^n$, which in view of Theorem A implies that the distribution of $X_n$, and thus also the normalized distribution of $\sum_{k \in \bigcup_{i=1}^n \Delta_i} f(n_k x)$, is close to the normal distribution. 
\color{black}
It remains to show that the contribution of the indices in the ``dashed'' blocks is indeed negligible, as claimed above.\\

We will use the fact that 
\begin{equation} \label{eq_exp}
e^x \leq 1 + x + x^2, \qquad \text{for all $x$ with } |x| \leq 1. 
\end{equation}
According to \eqref{eq_exp}, for all $\lambda$ with $|\lambda| \leq \frac{1}{\# \Delta_i'} = \frac{1}{6 \log_q N}$, writing $\exp(x) :=  e^x$, we have
\begin{eqnarray}
& & \int_0^1 \exp \left (\lambda \sum_{i=1}^n  \sum_{k \in \Delta_i'} \cos(2 \pi n_k x) \right) dx \nonumber\\
& = & \int_0^1 \prod_{i=1}^n \exp \left (\lambda  \sum_{k \in \Delta_i'} \cos(2 \pi n_k x) \right) dx \nonumber\\
&  \leq & \int_0^1 \prod_{i=1}^n \left( 1 + \lambda  \sum_{k \in \Delta_i'} \cos (2 \pi n_k x) + \left( \lambda  \sum_{k \in \Delta_i'}\cos (2 \pi n_k x) \right)^2 \right) dx. \label{eq_product}
\end{eqnarray}
We have
\begin{equation} \label{cos-square}
\left(\sum_{k \in \Delta_i'}\cos (2 \pi n_k x) \right)^2 = \frac{1}{2} \sum_{k,\ell \in \Delta_i'} \left( \cos (2 \pi (n_k+n_{\ell}) x) + \cos (2 \pi (n_k - n_{\ell} x))  \right).
\end{equation}
On the right-hand side of \eqref{cos-square} we have cosine-functions whose size of the frequency we control very efficiently. More precisely, $n_k+n_{\ell}$ is always in the range
$$
\left[ \min_{k \in \Delta_i'} n_k, ~2 \max_{k \in \Delta_i'} n_k  \right],
$$
and $n_k - n_\ell$ is (whenever it is non-zero, which happens exactly when $k = \ell$) of absolute value at least
$$
\left[ (q-1) \min_{k \in \Delta_i'} n_k, ~\max_{k \in \Delta_i'} n_k \right]
$$
as a consequence of the lacunary growth condition. Recall that the ``dashed'' blocks $\Delta_i'$ are separated by the long ``un-dashed'' blocks $\Delta_i$.  Thus the frequencies which we can encounter in \eqref{eq_product} and \eqref{cos-square} for some specific value of $i$ are separated by a large multiplicative factor from those which we could possibly encounter there for some other value of $i$. This is a classical argument in the theory of lacunary sums of dilated functions (worked out in detail for example in \cite[Lemma 2]{tak} and \cite[Lemma 3]{plt}), whose upshot is that in \eqref{eq_product} one can interchange integral and product signs (provided that $N$ is sufficiently large in comparison with the constant $q$), and that accordingly the expression in line \eqref{eq_product} equals
\begin{eqnarray*}
\prod_{i=1}^n \int_0^1 1 + \lambda \sum_{k \in \Delta_i'} \cos (2 \pi n_k x) + \left( \lambda  \sum_{k \in \Delta_i'}\cos (2 \pi n_k x) \right)^2 dx & = &  \prod_{i=1}^n \left( 1 + \frac{\lambda^2}{2}   \# \Delta_i'  \right) \\
& \leq & \exp \left( \frac{\lambda^2}{2} \sum_{i=1}^n \# \Delta_i' \right),
\end{eqnarray*}
where the last line follows from $1+x \leq e^x$. Overall this yields
$$
\int_0^1 \exp \left (\lambda \sum_{i=1}^n \sum_{k \in \Delta_i'} \cos(2 \pi n_k x) \right) dx \leq \exp \left( \frac{\lambda^2}{2} \sum_{i=1}^n \# \Delta_i' \right),
$$
for $|\lambda| \leq 1/(6 \log_q N)$. We set $\lambda = \frac{2 (\log N)^2}{\sqrt{N}}$. Then for sufficiently large $N$ we indeed have $|\lambda| \leq 1/(6 \log_q N)$, and we get
\begin{eqnarray*}
\mathbb{P} \left( x \in (0,1):~\sum_{i=1}^n \sum_{k \in \Delta_i'} \cos(2 \pi n_k x) > \frac{\sqrt{N}}{ \log N} \right) & \leq & \exp \left( \frac{\lambda^2}{2} \sum_{i=1}^n \# \Delta_i' - \frac{\lambda \sqrt{N}}{\log N} \right) \\
& \leq & \exp \left( 2\frac{(\log N)^4}{N} \frac{6 N \log_q N}{(\log N)^5} - 2 \log N \right) \\
& = & \exp \left(\frac{12}{\log q} - 2  \log N \right) \\
& = & O (1/N^2).
\end{eqnarray*}
The same argument can be carried out for $-\sum_{i=1}^n \sum_{k \in \Delta_i'} \cos(2 \pi n_k x) $ instead of $\sum_{i=1}^n \sum_{k \in \Delta_i'} \cos(2 \pi n_k x) $, so that we can insert absolute values and obtain
$$
\mathbb{P} \left( x \in (0,1)\,:\, \left|\sum_{i=1}^n \sum_{k \in \Delta_i'} \cos(2 \pi n_k x)  \right| > \frac{\sqrt{N}}{ \log N} \right) = O \left(1/N^2 \right).
$$
Thus, since $|f(x)| \leq \sum_{j=1}^d \left| c_j \cos(2 pi j x) \right|$,  overall we obtain
\begin{equation} \label{eq_***}
\left|\sum_{i=1}^n \sum_{k \in \Delta_i'} f(n_k x) \right| = O \left( \frac{\sqrt{N}}{\log N} \right), 
\end{equation}
except on a set of probability at most $O(1/N^2)$.\\

By Theorem A, for any fixed $\eta \in (0,1)$ we have 
\begin{equation} \label{eq_tha_conclusion}
\frac{\mathbb{P} (X_n \geq r)}{1 - \Phi(r)} \to 1,
\end{equation}
uniformly in $r$ subject to $0 \leq r \leq \sqrt{2 \eta \left| \log \left( L_4^n + N_4^n \right) \right|}$. In our setup, by \eqref{eq_L_4_est} and \eqref{eq_N_4_est} we have
$$
\left| \log \left( L_4^n + N_4^n \right) \right| \geq \log \left( \min \left( 1/L_4^n, 1/N_4^n \right) \right) \geq   \log \left( \min \left( N / (\log N)^{15}, g_N \right) \right) \geq \eta \log g_N, 
$$
since by assumption $g_N \leq N$ and thus  $\min \left( N / (\log N)^{15}, g_N \right) \geq g_N^\eta$. Thus, \eqref{eq_tha_conclusion} holds uniformly for $0 \leq r \leq  \sqrt{2 \eta^{3/2} \log g_N}$.\\

Recall that by \eqref{eq_x_dist}
$$
\left| X_n  - \frac{1}{\|f\|_2 \sqrt{\sum_{i=1}^n \#\Delta_i}} \sum_{i=1}^n \sum_{k \in \Delta_i} f(n_k x) \right| = O \left(N^{-5/2} \right). 
$$
Furthermore, we saw in \eqref{eq_**} that
$$
N - \sum_{i=1}^n \#\Delta_i = O \left(N / (\log N)^4 \right),
$$
and established in \eqref{eq_***} that 
$$
\left| \sum_{k \in \Delta_i'} f(n_k x) \right| = O \left( \frac{\sqrt{N}}{\log N} \right), 
$$
except on a set of probability at most $O(1/N^2)$. Note also  that we consider only $r$ in the range $0 \leq r \leq \sqrt{2 \eta^{3/2} \log g_N}$, where $g_N \leq N$ by assumption, so that 
$$
1 - \Phi(r) \geq 1/N
$$
throughout the relevant range, for $N\in\mathbb N$ sufficiently large. All these facts together allow the application of a simple approximation argument, which permits us to replace first $X_n$ by 
\[
  \frac{1}{\|f\|_2 \sqrt{\sum_{i=1}^n \#\Delta_i}} \sum_{i=1}^n \sum_{k \in \Delta_i} f(n_k x),
\]
then replace the normalization factor $\sqrt{\sum_{i=1}^n \# \Delta_i}$ by $\sqrt{N}$, and finally replace $\sum_{i=1}^n \sum_{k \in \Delta_i} f(n_k x)$ by $\sum_{k=1}^N f(n_k x)$ by inserting also the contribution of the ``dashed'' blocks. After this procedure (whose details we omit) we arrive at
\begin{equation*}
\frac{\mathbb{P} \left( x \in (0,1): \frac{1}{\|f\|_2 \sqrt{N}} \sum_{k=1}^N  f(n_k x) \geq r \right)}{1 - \Phi(r)} \to 1,
\end{equation*}
uniformly for $0 \leq r \leq  \sqrt{2 \eta^{3/2} \log g_N}$. For given $\varepsilon >0$, we can choose $\eta>0$ so close to 1 that  $\eta^{3/4} \geq (1-\varepsilon)$, such that the desired result holds throughout the range $0 \leq r \leq  (1-\varepsilon) \sqrt{2 \log g_N}$. This concludes the proof of Theorem \ref{thm_iid}.

\section{Proof of Theorem \ref{thm_not_iid}}

Assume that we are given a sequence $(g_N)_{N \in \mathbb{N}}$ as in the statement of Theorem \ref{thm_not_iid}, and a number $\varepsilon>0$. We will construct a lacunary sequence $(n_k)_{k \geq 1}$ and a trigonometric polynomial $f$ such that the number of solutions of the relevant Diophantine equations satisfies the upper bound  $L(N,a,b)= O\left( \frac{N}{g_{N}} \right)$, but such that the size of the tails of the distribution of the normalized lacunary sum $S_N$ at deviation level $(1+\varepsilon) \|f\|_2 \sqrt{2 \log g_N}$ fails to coincide with the corresponding quantity for the standard normal distribution, for infinitely many $N$. 

\subsection{Construction of the sequence}

Generally speaking, the construction of the sequence for the proof of Theorem \ref{thm_not_iid} is in the spirit of a construction used in \cite{AFP_Lil}, but several modifications are necessary to account for the presence of the sequence $(g_N)_{N \in \mathbb{N}}$ in the statement of Theorem \ref{thm_not_iid} (which in the setup of \cite{AFP_Lil} was fixed to be $g_N = (\log N)^{1-\varepsilon},~N \geq 1$, thereby allowing a much more rigid construction of the sequence). We start by defining
\[
\Delta_i:= \left\{ 2^{3^i-1}, \ldots, 2^{3^{(i+1)}-1} -1\right\}, \quad i \in \mathbb{N}_{0}.
\]
We then set $M(i):= \lceil g_{ \# \Delta_i } \rceil$ and subdivide $\Delta_i$ into shorter blocks $\Delta_i^{(m)}$, for $m=1, \ldots, M(i)$. This can be done in a way such that all $\Delta_i^{(m)}$ have, up to an error of $1$, the same cardinality which essentially equals $ \frac{ \# \Delta_i}{M(i)} $. By construction, all these sets together form a partition of $\N$, i.e.
\[
\N = \bigcupdot_{i \in \N_0} \bigcupdot_{m=1}^{ M(i)} \Delta_i^{(m)}.
\]
We now define our sequence $(n_k)_{k \in \N}$ as
\[
n_k := 2^{2^{2^{2^{i}}}} \left( 2^{k+im} +m \right) , \quad k \in \Delta_i^{(m)} .
\]
We briefly describe the heuristics behind this construction. The fourfold exponential term $2^{2^{2^{2^{i}}}}$ here is somewhat arbitrary; its main purpose is to ensure that there is a massive jump in the size of the sequence elements $n_k$ when $k$ moves from $\Delta_i$ to $\Delta_{i+1}$. Similarly, the ``$+im$'' component in the term $2^{k+im}$ leads to a jump in the size of $n_k$ when $k$ moves from some $\Delta_i^{(m)}$ to $\Delta_i^{(m+1)}$. These effects are used to establish stochastic independence between sums $\sum_{k \in \Delta_i^{(m)}} f(n_k x)$ for different values of $m$ or of $i$. Finally, the additive ``$+m$'' term in the definition of the sequence ensures that there are many (but not too many) solutions of Diophantine equations such as
\begin{equation} \label{many_sol}
n_{k+1} - 2 n_k = c, \qquad k \in \Delta_i^{(m)}
\end{equation}
for suitable values of $c$ (namely integer multiples of $m$ times the common factor $2^{2^{2^{2^{i}}}}$). The construction is made in such a way that the set of those values of $c$ on the right-hand side of \eqref{many_sol} for which we find many solutions of the Diophantine equations change as $m$ and $i$ changes. Thus, the total number of solutions of the Diophantine equation is bounded above, independent of the actual value of the number $c$ on the right-hand side, by the relative proportion of indices contained in a sub-block $\Delta_i^{(m)}$, compared to the overall number of indices. By construction $\# \Delta_1 + \dots + \# \Delta_i \approx \# \Delta_i$, and writing $N$ for $\# \Delta_1 + \dots + \# \Delta_i$, the number of solutions of the relevant Diophantine equations is thus bounded above by $\# \Delta_i^{(m)} \approx \# \Delta_i / M(i) \approx N / g_N$, as desired.\\

To begin with, we prove that the sequence we just  constructed is indeed a lacunary sequence.

\begin{lem}
\label{lem:seq_is_lacunary}
For the integer sequence $(n_k)_{k \in \N}$ constructed above, we have
\[
\liminf_{k \rightarrow \infty} \frac{n_{k+1}}{n_k} = 2.
\]
\end{lem}
\begin{proof}
If $k, k+1 \in \Delta_i^{(m)}$ for some values of $m$ and $i$, then we have
\begin{equation} \label{case_studied}
\frac{n_{k+1}}{ n_k} = \frac{2^{k+1+im}+m}{2^{k+im}+m} = 2 \frac{1+ 2^{-k-im-1} m}{1+ 2^{-k-im} m} \asymp 2,
\end{equation}
for sufficiently large $k$ (i.e., for sufficiently large $i$, since $i$ and $k$ increase together).\\

Now assume that $k$ is the last element in $ \Delta_i^{(m)}$ for some pair $(i,m)$ and $k+1$ is the first element in $\Delta_i^{(m+1)}$. Then we have
\begin{equation} \label{middle}
\frac{n_{k+1}}{n_k} = \frac{2^{k+1+i(m+1)}+m+1}{2^{k+im} + m} \geq 2^{i+1}  \frac{1+ 2^{-k-im-1} m}{1+ 2^{-k-im} m} \asymp 2^{i+1}
\end{equation}
for sufficiently large $k$.\\

Finally, assume that $k$ is the last element in $ \Delta_i^{(M(i))}$ for some pair $(i,M(i))$ and $k+1$ is the first element in $\Delta_{i+1}^{(1)}$. Then we have
\begin{equation} \label{cont_here}
\frac{n_{k+1}}{n_k} = \frac{2^{2^{2^{2^{i+1}}}}}{2^{2^{2^{2^i}}}} \frac{2^{k+1+i + 1}+1}{2^{k+i M(i)} + M(i)} \geq \frac{2^{2^{2^{2^{i+1}}}}}{2^{2^{2^{2^i}}}} \frac{2^{k+1+i + 1}}{2^{k+i M(i)+1}} \geq \frac{2^{2^{2^{2^{i+1}}}}}{2^{2^{2^{2^i}}} 2^{i M(i)}}.
\end{equation}
Since $g_N \leq N$ for all $N$ by assumption, we have $M(i) =  \lceil g_{ \# \Delta_i } \rceil \leq \# \Delta_i \leq 2^{3^{i+1}}$. Thus in the present case, continuing from \eqref{cont_here}, we have
$$
\frac{n_{k+1}}{n_k} \geq \frac{2^{2^{2^{2^{i+1}}}}}{2^{2^{2^{2^i}}} 2^{i 2^{3^{i+1}}}},
$$
and one can easily check that this diverges to infinity as $i$ increases. In conclusion, asymptotically the smallest quotients of $n_{k+1}/n_k$ come from the case studied in \eqref{case_studied}, and we have 
$$
\liminf_{k \to \infty} \frac{n_{k+1}}{n_k} = 2
$$
as claimed.
\end{proof}

\begin{rem}
\label{rem:big_gaps}
During the previous proof, we have seen that the \glqq{}gaps\grqq{} (in terms of the size of $(n_k)_{k \in \N}$) between two elements of our sequence with consecutive indices in $\Delta_i^{(m)}$ and in $\Delta_i^{(m+1)}$, or in $\Delta_i^{(M(i))}$ and in $\Delta_{i+1}^{(1)}$, become arbitrarily large as $i$ becomes large. As we will see below, this is the property we need for the lacunary partial sums over indices in different blocks to behave stochastically independent of each other.
\end{rem}

\begin{lem}
\label{lem:independent_blocks}
Let $(n_k)_{k \in \N}$ be a lacunary sequence with growth factor $\eta > 1$. Let $(\Lambda_r)_{r \in \N}$ be a sequence of finite subsets of $\N$ with the following property: For all $M > 0$, there exists an $r_0 \in \N$ such that for all $r \geq r_0$, we have
\[
\frac{\min_{k \in \Lambda_{r+1}} n_k }{\max_{k \in \Lambda_{r}} n_k} \geq M.
\]
Then, for all $a,b \in \N$, it holds that
\[
\sup_{c \in \Z} \sum_{\substack{r,s =1 \\ r \neq s}}^{\infty} \# \big\{ k \in \Lambda_r, \ell \in \Lambda_s \,:\, a n_k - b n_{\ell} = c \big\} = O(1),
\]
with the implicit constant in the $O$-term only depending on $a,b,r_0,$ and $\eta$.
\end{lem}

\begin{proof}
We can assume that $r > s$, since the other case works entirely symmetric. Let $M = \frac{2\eta}{\eta-1}  > 0$ and let $r_0$ be the index with 
\[
\frac{\min_{k \in \Lambda_{r+1}} n_k }{\max_{k \in \Lambda_{r}} n_k} \geq \frac{b M}{a}
\]
for all $r \geq r_0$. In the following, we only consider the case $r \geq r_0$, since this excludes only finitely many Diophantine solutions. Let $k \in \Lambda_r$ and $\ell \in \Lambda_s$, then by the choice of $M$ the equality
\[
a n_k - bn_\ell = c
\]
can only hold for non-negative $c$. Moreover, we get
\[
\frac{an_k}{bn_\ell} \geq M
\]
and hence
\[
a n_k \left( 1 - \frac{1}{M} \right) = a n_k - \frac{a n_k}{M} \leq a n_k - b n_{\ell} \leq a n_k.
\]
This means any solution to the equation $a n_k - bn_\ell = c$ requires that
\begin{equation}
\label{eq:nk_unique}
n_k \in \left [ \frac{c}{a}, \left( 1- \frac{1}{M} \right)^{-1} \frac{c}{a} \right ].
\end{equation}

We observe that by the choice of $M$, for fixed $c \in \N$, there is at most one $r \geq r_0$ and at most one $k \in \Lambda_r$ such that \eqref{eq:nk_unique} holds. This is because if 
\[
\frac{c}{a} \leq n_k \leq \left( 1- \frac{1}{M} \right)^{-1} \frac{c}{a},
\]
then $n_{k+1} \geq \eta \frac{c}{a} > \left( 1- \frac{1}{M} \right)^{-1} \frac{c}{a}$, by the choice of $M$. Vice versa, $n_{k-1} \leq \frac{1}{\eta} n_k \leq \frac{1}{\eta} \left( 1- \frac{1}{M} \right)^{-1} \frac{c}{a} < \frac{c}{a}$. Now for fixed $k$ and $c$, there is of course at most one possible value of $\ell$ such that $a n_k - b n_\ell = c$. This shows that
\[
\sup_{c \in \N} \sum_{r = r_0}^{\infty} \sum_{s=1}^{r-1} \big\{ k \in \Lambda_r, \ell \in \Lambda_s \,:\, a n_k - b n_\ell = c \big\} = O(1).
\]
\end{proof}

\begin{lem}
\label{lem:Count_Dio}
The integer sequence $(n_k)_{k \in \N}$ constructed in this section satisfies the Diophantine condition in Equation \eqref{Eq_Cond_Dio} of Theorem \ref{thm_not_iid} for all $N\in\mathbb N$ of the form $N=\max \Delta_i, ~i \in \N$.
\end{lem}

\begin{proof}
Let $(N_i)_{i \in \N}$ be the sequence with $N_i = \max \Delta_i$ for $i \in \N$. Then, for all $a,b \in \N$, we need to show that
\[
L(N_I,a,b) = \sup_{c \in \Z} \sum_{k,\ell =1}^{N_I} \mathbb{1}_{[a n_k - b_{\ell} = c]} = O \left( \frac{N_I}{g_{N_I}} \right), \quad I \in \N.
\]
Let $(\Lambda_r)_{r \in \N}$ be the sequence of blocks $\Delta_1^{(1)}, \Delta_1^{(2)}, \ldots, \Delta_1^{(M(1))}, \Delta_2^{(1)}, \Delta_2^{(2)}, \ldots, \Delta_2^{(M(2))}, \ldots $. Note that, by Remark \ref{rem:big_gaps}, the assumptions of Lemma \ref{lem:independent_blocks} are satisfied. This means, we have
\[
L(N_I,a,b) \leq \sup_{c \in \Z } \sum_{i=1}^I \sum_{m=1}^{M(i)} \# \left\{ k, \ell \in \Delta_i^{(m)} \,:\, a n_k - b n_\ell = c \right\} + O(1).
\]
Now we give an upper bound for
\[
\sup_{c \in \Z } \sum_{i=1}^I \sum_{m=1}^{M(i)} \# \left\{ k, \ell \in \Delta_i^{(m)} \,:\, a n_k - b n_\ell = c \right\}.
\]
Let us first assume that $\frac{a}{b}$ is not a power of $2$, then analogously to the argument in \cite[p.555--556]{AFP_Lil}, it follows that the quantity above is bounded by an absolute constant. In \cite{AFP_Lil} the sequence $(n_k)_{k \geq 1}$ was defined in similar fashion as in the present paper, but with $n_{k} = 2^{2^{i^4}} \left(2^k+m \right)$ instead of $n_k = 2^{2^{2^{2^{i}}}} \left( 2^{k+im} +m \right)$  for $k \in \Delta_i^{(m)}$; however, this does not affect the argument in any significant way. The only relevant ingredient for this part of the argument is that $n_{k+1}/n_k$ approaches either $2$ or $+\infty$ as $k$ increases, which is exactly what we have established in \eqref{case_studied}--\eqref{cont_here}; accordingly, since by assumption $\frac{a}{b}$ is not a power of $2$, the sequence formed as the set-theoretic union of $(a n_k)_{k \geq 1}$ and $(b n_k)_{k \geq 1}$ (sorted in increasing order) is also a lacunary sequence, and one can apply a general result on the maximal number of representations of integers as a difference of elements of a lacunary sequence (namely the result which appears for example on p.\ 203 of Zygmund's book \cite{zyg_book}). The same argument applies also to the case when $a=b$.\\

Let us thus assume $ \frac{a}{b} = 2^r $ for some $r \in \Z$, $r \neq 0$. We can assume in the sequel that $r \geq 1$ (since for $r \leq -1$ we can simply switch the roles of $k$ and $\ell$). Then for any $k, \ell \in \Delta_i^{(m)}$ we have $a n_k - b n_\ell = c$ when
\begin{equation} \label{assuming_this}
2^{k+r} - 2^\ell = \frac{c }{2^{2^{2^{2^{i}}}} b 2^{mi}} - \frac{(2^r-1)m}{2^{im}}.
\end{equation}
In the particular case of $c=0$, for $k, \ell \in \Delta_i^{(m)}$ we have $a n_k - b n_\ell = 0$ when
\begin{equation} \label{imposs}
2^{k+r} - 2^\ell = - \frac{(2^r-1)m}{2^{im}};
\end{equation}
it is clear that this is impossible when $i$ is sufficiently large, since the left-hand side is always an integer while the right-hand side goes to zero as $i \to \infty$, uniformly in $m$). Thus it remains to study the case when $c \neq 0$. In this case, when $k, \ell \in \Delta_i^{(m)}$ and $a n_k - b n_\ell = c$, then this implies
\begin{equation} \label{this_says_1}
|c| \leq \max (a,b) \max_{k \in \Delta_i^{(m)}} n_k.
\end{equation}
Furthermore, when $k+r > \ell$, then $a n_k - b n_\ell = c$ implies
\begin{eqnarray}
c  & = & 2^{2^{2^{2^{i}}}} \left(2^r b \left( 2^{k + im} + m \right) - b \left( 2^{\ell + im} + m \right) \right) \nonumber\\
& \geq & 2^{2^{2^{2^{i}}}} \left(b \left( 2^{\ell +1 + im} + 2^r m \right) - b \left( 2^{\ell + im} + m \right) \right) \nonumber\\
& \geq & 2^{2^{2^{2^{i}}}} b \left( 2^{\ell + im} + m \right) \nonumber\\
& \geq & b \min_{\ell \in \Delta_i^{(m)}} n_\ell \nonumber\\
& \geq & \min_{\ell \in \Delta_i^{(m)}} n_\ell. \label{this_says_2}
\end{eqnarray}
Similarly, when $k+r < \ell$, then $a n_k - b n_\ell = c$ implies
\begin{eqnarray}
c  & = & 2^{2^{2^{2^{i}}}} \left(2^r b \left( 2^{k + im} + m \right) - b \left( 2^{\ell + im} + m \right) \right) \nonumber\\
& \leq & 2^{2^{2^{2^{i}}}} \left(b \left( 2^{k+ r + im} + 2^r m \right) - b \left( 2^{k+r+1 + im} + m \right) \right) \nonumber\\
& \leq & - 2^{2^{2^{2^{i}}}} b \left( 2^{k + r + im} - \left(2^r - 1 \right) m \right) \nonumber\\
& \leq &  -b 2^r \left( \min_{k \in \Delta_i^{(m)}} n_k  - \left(2^{r+1} - 1 \right) m \right) \nonumber\\
& \leq & - \min_{k \in \Delta_i^{(m)}} n_k \label{this_says_3}
\end{eqnarray}
for sufficiently large $i$ (since $r \geq 1$ is fixed, and $m$ is asymptotically much smaller than $n_k$ for $n_k \in \Delta_i^{(m)}$). What \eqref{this_says_1}, \eqref{this_says_2}, \eqref{this_says_3} say, together with \eqref{middle} and \eqref{cont_here}, is the following: for fixed $a$ and $b$ with $\frac{a}{b} = 2^r$, for all $i_1,i_2,m_1,m_2$ with $i_1$ and $i_2$ sufficiently large, when we asssume that $(i_1,m_1) \neq (i_2,m_2)$, then the two sets
$$
\left\{c \in \mathbb{Z}\,:\,\exists k,\ell \in \Delta_{i_1}^{(m_1)},k+r \neq \ell:~a n_k - b n_\ell = c \right\} \quad \text{and} \quad \left\{c \in \mathbb{Z}\,:\,\exists k,\ell \in \Delta_{i_2}^{(m_2)},k+r \neq \ell:~a n_{k} - b n_\ell = c \right\}
$$
are disjoint.\\

Now we come back to assuming that for some $k, \ell \in \Delta_i^{(m)}$ we have $a n_k - b n_\ell = c$, so that \eqref{imposs} is satisfied. If $k + r > \ell$, then by the uniqueness of the $2$-adic expansion, there is at most one pair $(k,\ell)$ for which \eqref{imposs} can hold, for any given configuration of $a,b,c,m,i$. The same applies in the case $k + r < \ell$. Thus for any given $a,b,c,m,i$ we have 
$$
\sup_{c \in \Z } \# \left\{ k, \ell \in \Delta_i^{(m)}, k + r \neq \ell \,:\, a n_k - b n_\ell = c \right\} = O(1).
$$
However, by what was said in the paragraph after equation \eqref{this_says_3}, as a consequence of \eqref{middle}--\eqref{cont_here} and \eqref{this_says_1}--\eqref{this_says_3} this remains true when we take a summation over $i$ and $m$, so that overall we have
\[
\sup_{c \in \Z } \sum_{i=1}^I \sum_{m=1}^{M(i)} \# \left\{ k, \ell \in \Delta_i^{(m)}, k + r \neq \ell \,:\, a n_k - b n_\ell = c \right\} = O(1).
\]

If now $k + r = \ell$, then for given $i$ and given $m \leq M(i)$ there is precisely one $c \in \Z$ such that $2^{k+r} - 2^\ell = 0 = \frac{c }{2^{2^{i^4}}b} - (2^r-1)m$. For this particular value of $c$ (which depends on $i$ and $m$), all the pairs $(k,\ell) \in \Delta_i^{(m)}$ with $k+r = \ell$ solve $a n_k - b n_\ell = c$, and the number of such pairs is 
$$
\# \Delta_i^{(m)} - r \leq \# \Delta_i^{(m)}.
$$
Thus we have
\begin{align*}
\sup_{c \in \Z } \sum_{i=1}^I \sum_{m=1}^{M(i)} \# \left\{ k, \ell \in \Delta_i^{(m)}, k + r = \ell \,:\, a n_k - b n_\ell = c \right\} & = O \left(\max_{i \le I}  \max_{m \leq M(i)} \# \Delta_i^{(m)} \right) \\
& = O \left( \frac{N_I}{g_{N_I}} \right),
\end{align*}
where the last equality follows from our particular choice of the size of the blocks $\Delta_i$ and sub-blocks $\Delta_i^{(m)}$. This finishes the proof.
\end{proof}

During the proof of Theorem \ref{thm_not_iid}, we will use the following quantitative CLT for lacunary pure trigonometric sums, due to Gaposhkin.

\begin{lem}(Gaposhkin \cite{gapo_Berry_Esseen})
\label{lemma_gaposhkin}
Let $\lambda_1, \dots, \lambda_N$ be non-negative real numbers such that 
$$
\sum_{k=1}^N \lambda_k^2 = 1. 
$$
Set $\Lambda_N = \max_{1 \leq k \leq N} \lambda_k$. Then
$$
\sup_{t \in \mathbb{R}}\left| \lambda\left[ x \in (0,1),\, \sqrt{2} \sum_{k=1}^N \lambda_k \cos( 2\pi 2^k x) < t \right] - \Phi(t) \right| = O \left(\Lambda_N^{1/4} \right), 
$$
where the implied constant is absolute.
\end{lem}





\begin{proof}[Proof of Theorem \ref{thm_not_iid}]
In this proof, we work on the probability space $([0,1], \mathcal{B}\left( [0,1] \right), \mathbb{P})$, where $\mathbb{P} = \lambda$ is the Lebesgue-measure on $[0,1]$. We will also repeatedly make use of the following elementary estimate: Let $ A, B \in \mathcal{B}\left( [0,1] \right)$, then 
\[\mathbb{P} \left[ A \cap  B \right] \geq \mathbb{P} \left[ A \right]- \mathbb{P} \left[ B^c\right].
\]
\par{}
Let $(N_i)_{i \in \N}$ be the sequence of integers defined by $N_i := \max \Delta_i  = 2^{3^{i+1}-1}-1$. Let $d$ be some fixed positive integer, to be specified later, and define 
$$
f(x) := \sum_{j=0}^{d-1} \cos( 2 \pi 2^j x).
$$
In the following, we will give a lower bound for
\[
\mathbb{P} \left[  \sum_{k=1}^{N_I} f( n_k x)  \geq  (1+\varepsilon) \| f \|_2 \sqrt{2 N_I \log g_{N_I}} \right],
\]
when $I \in \N$ is large. We can write
\begin{align*}
\sum_{k=1}^{N_I} f( n_k x) &= \sum_{i=1}^I \sum_{k \in \Delta_i} f(n_k x) =: \sum_{i=1}^I Y_i.
\end{align*}
We note that $ \| \sum_{i=1}^{I-1} Y_i \|_{\infty} \leq d \sum_{i=1}^I 2^{3^i} = o\left( 2^{(3^{I+1})/2} \right) $. Thus, for any fixed $\delta > 0$, if $I \in \N$ is sufficiently large, we have $ \| \sum_{i=1}^{I-1} Y_i \|_{\infty} \leq \delta  \sqrt{ 2 N_I \log g_{N_I}} $, since $\log g_{N_I}$ tends to infinity as $I$ tends to infinity. Hence we get
\begin{align}
\label{eq:decomp_YI}
\mathbb{P} \left[ \sum_{i=1}^{I} Y_i \geq (1 + \varepsilon)\| f\|_2 \sqrt{2 N_I \log g_{N_I}} \right] & \geq \mathbb{P} \left[  Y_I \geq ( (1 + \varepsilon) \| f\|_2 + \delta) \sqrt{2 N_I \log g_{N_I}} \right],
\end{align}
and it remains to study the distribution of $Y_I$ (this is why in our construction, the size of the blocks $\Delta_i,~i \geq 1,$ was chosen in a quickly increasing way). 
Now, using standard trigonometric identities and arguing as in \cite[p.~566--567]{AFP_Lil},  the distribution of $Y_I$ is the same as that of
\begin{align*}
& d \sum_{m=1}^{M(I)} \sum_{k \in \Delta_I^{(m)}}  \cos(2 \pi 2^{k+I m} x) - \sum_{m=1}^{M(I)} \sum_{j=0}^{d-1} \left( \sin(2 \pi 2^j m x) \right)^2  \sum_{k \in \Delta_I^{(m)}} \cos(2 \pi 2^{k+I m} x ) \\
 & \qquad - \sum_{m=1}^{M(I)} \sum_{j=0}^{d-1} \sin(2 \pi 2^j m x)   \sum_{k \in \Delta_I^{(m)}} \sin(2 \pi 2^{k+Im} x) +  E_{I}(x),
\end{align*}
where $E_{I}$ is an error term of the form $\sum_{m=1}^{M(I)} \sum_{k \in \Xi_I^{(m)}} f((2^{k+I m} + m) x)$, with $\Xi_I^{(m)}$ denoting the set of all those values of $k$ which are among either the first or the last $d$ many indices in $\Delta^{(m)}_{I}$ for some $m \in \{1, \dots, M(I)\}$ (for which there does not exist the ``correct'' number of other digits $\ell \in \Delta_I^{(m)}$ with $|k - \ell| \leq d$; see the detailed argument in \cite[p.~566]{AFP_Lil}). For this error term we have
$$
\|E_{I}\|_\infty \leq 2d^2 M(I),
$$
so if $g_{N_I} \leq N_I^{1/4}$, then we have
\begin{equation} \label{size_of_E_small}
\left\|E_{I} \right\|_\infty = O \left( N_I^{1/4} \right).
\end{equation}
On the other hand, if $g_{N_I} \geq N_I^{1/4}$, then writing 
\begin{equation} \label{right-hand}
E_{I}(x) = \sum_{j=0}^{d-1} \sum_{m=1}^{M(I)} \sum_{k \in \Xi_I^{(m)}} \cos(2 \pi 2^j (2^{k+Im}+m) x),
\end{equation}
we use the fact that for each fixed $j$ the sum on the right-hand side of \eqref{right-hand} is a lacunary sum of $2d M(I)$ many summands, where $2dM(I) = \Omega(N_I^{1/4})$ and $2dM(I) = o (N_I)$ (recall here that $M(I) = O(g_{N_I}) = o(N_I)$). Applying \cite[Corollary 2]{AFHM_mgf}, for fixed $\delta > 0$ we obtain
$$
\mathbb{P} \left[  \left|\sum_{m=1}^{M(I)} \sum_{k \in \Xi_I^{(m)}} \cos(2 \pi 2^j (2^{k+Im}+m) x) \right| \geq \frac{\delta}{d}\sqrt{ 2 N_I \log g_{N_I}} \right] = O \left(N_I^{-1/16} \right),
$$
say. Accordingly, by \eqref{right-hand} we also have 
\begin{equation} \label{size_of_E}
\mathbb{P} \left[  \left|E_{I}(x) \right| \geq \delta \sqrt{ 2 N_I \log g_{N_I}} \right] = O \left( N_I^{-1/16} \right),
\end{equation}
still under the assumption that $g_N \geq N_I^{1/4}$. However, we note that \eqref{size_of_E_small} implies \eqref{size_of_E} for any (fixed) choice of $\delta>0$, so that \eqref{size_of_E} actually holds without any assumptions on the size of $g_N$. Accordingly the contribution of the error term $E_I$ is ``small'', and in the sequel we will concentrate on the distribution of the main term, which is $Z_I := Y_I - E_I$.\\

Let $h_I \in \N$ be such that 
\begin{equation*}
  \frac{1}{20 d^2 2^d M(I)} \leq   2^{-h_I} \leq   \frac{1}{10 d^2 2^d M(I)}.
\end{equation*}
Then,
\begin{align}
\notag
&  \mathbb{P} \left[ Z_I \geq \left( (1 + \varepsilon) \| f\|_2 + 2 \delta \right) \sqrt{2 N_I \log g_{N_I}} \right] \\
\notag
&  \geq \mathbb{P} \left[ x \in [0,2^{-h_I}], Z_I \geq \left( (1 + \varepsilon) \| f\|_2 + 2 \delta \right) \sqrt{ 2 N_I \log g_{N_I}} \right]  \\
\label{Eq_dominating_term}
&  \geq \mathbb{P} \left[ x \in [0,2^{-h_I}], d \sum_{m=1}^{M(I)} \sum_{k \in \Delta_I^{(m)}} \cos(2 \pi 2^{k+I m} x) \geq \left( (1 + \varepsilon) \|f\|_2 + 4\delta  \right) \sqrt{2 N_I \log g_{N_I}} \right] \\
\label{Eq_sin_squared_term}
& \qquad  - \mathbb{P} \left[ x \in [0,2^{-h_I}], \left | \sum_{m=1}^{M(I)} \sum_{j=0}^{d-1} \left( \sin(2 \pi 2^j m x) \right)^2  \sum_{k \in \Delta_I^{(m)}} \cos(2 \pi 2^{k+I m} x ) \right| > \delta \sqrt{2 N_I \log g_{N_I}} \right] \\
\label{Eq_sin_term}
& \qquad  - \mathbb{P} \left[ x \in [0,2^{-h_I}], \left | \sum_{m=1}^{M(I)} \sum_{j=0}^{d-1} \sin(2 \pi 2^j m x)   \sum_{k \in \Delta_I^{(m)}} \sin(2 \pi 2^{k+I m} x) \right| > \delta \sqrt{2 N_I \log g_{N_I}} \right].
\end{align}
We first consider the term in \eqref{Eq_dominating_term}. Using periodicity and setting $ s= s(\varepsilon, \delta, I) :=  \Big ( (1 + \varepsilon) \|f\|_2 + 4\delta \Big )$, we get 
\begin{align}
\notag
 & \mathbb{P} \left[ x \in [0,2^{-h_I}], d \sum_{m=1}^{M(I)} \sum_{k \in \Delta_I^{(m)}}  \cos(2 \pi 2^{k+I m} x) \geq s \sqrt{2 N_I \log g_{N_I}} \right] \\
 \notag & = 2^{-h_I} \mathbb{P} \left[ d \sum_{m=1}^{M(I)} \sum_{k \in \Delta_I^{(m)}}  \cos(2 \pi 2^{k+I m} x) \geq s \sqrt{2 N_I \log g_{N_I}} \right] \\
  \notag
  &= 2^{-h_I} \left( \exp \left( -\frac{2s^2}{ d^2} \frac{N_I}{\# \Delta_I} \log g_{N_I} \right) + O \left( \# \Delta_I^{-1/4}\right) \right) \\
  \label{Eq_Est_dominating_term}
  &= 2^{-h_I} \exp \left( -\frac{2s^2}{d^2} \frac{N_I}{\# \Delta_I} \log g_{N_I} \right) \left(  1 + o\left( 1\right) \right),
\end{align}
from Gaposhkin's quantitative CLT (Lemma \ref{lemma_gaposhkin}). The last estimate in \eqref{Eq_Est_dominating_term} holds since $-\frac{2s^2}{d^2} \frac{N_I}{\# \Delta_I} \leq \frac{1}{10}$ for sufficiently large $d \in \N$ and thus the exponential function asymptotically dominates the $O(\# \Delta_I^{-1/4})$-term.\\

We now need to show that the terms in \eqref{Eq_sin_squared_term} and \eqref{Eq_sin_term} do not make a relevant contribution as $I \rightarrow \infty$.

Recall that by construction the smallest index $k$ in $\Delta_I$ is of size at least $2^{3^I}$, and that for sufficiently large $I \in\mathbb N$, we have $2^{3^I} \geq h_I$. We will work on short intervals of the form $\left[\frac{a}{2^{4^{I-1}}},\frac{a+1}{2^{4^{I-1}}} \right] \subset [0,2^{-h_I}]$ for suitable integers $a$. Within such an interval, the function $\sum_{j=0}^{d-1} (\sin (\pi 2^j m x))^2$ is essentially constant. More precisely, writing
$$
s_{a,m,I} := \sum_{j=0}^{d-1} \left(\sin \left(\pi 2^j m \frac{a}{2^{4^{I-1}}}\right)\right)^2,
$$
by considering derivatives, we obtain the Lipschitz estimate
\begin{equation} \label{sami_discrete}
\left| \sum_{j=0}^{d-1} (\sin (\pi 2^j m x))^2 - s_{a,m,I} \right| \leq \frac{2^{d+1} \pi m}{2^{4^{I-1}}} \quad \text{uniformly for $x \in \left[\frac{a}{2^{4^{I-1}}},\frac{a+1}{2^{4^{I-1}}} \right]$}.
\end{equation}

Furthermore, using the estimate $\sin(x) \leq x$ for small $x$, we have
\begin{eqnarray}
s_{a,m,I} & \leq & \sum_{j=0}^{d-1} \left(\pi 2^j m \frac{a}{2^{4^{I-1}}} \right)^2 \nonumber\\
& \leq & 4 d 2^d M(I)^2 \left( \frac{a}{2^{4^{I-1}}} \right)^2. \label{saim}
\end{eqnarray}
For $I \in\mathbb N$, we set
$$
S_{a,I} := \left(\sum_{m=1}^{M(I)}  \sum_{k \in \Delta_I^{(m)}} s_{a,m,I}^2 \right)^{1/2},
$$
and define
$$
\lambda_k := \frac{s_{a,m,I}}{S_{a,I}}, \qquad \text{for $k \in \Delta_I^{(m)}$.}
$$
Then, we clearly have
$$
\sum_{k \in \Delta_I} \lambda_k^2 = 1.
$$
On the other hand, using the estimate $\sin(x) \geq \frac{x}{2}$ for small values of $x$, we obtain
\begin{align*}
S_{a,I} & = \left(\sum_{m=1}^{M(I)}  \sum_{k \in \Delta_I^{(m)}} s_{a,m,I}^2 \right)^{1/2} \\
& = \# \Delta_I^{1/2} \left( \frac{1}{M(I)} \sum_{m=1}^{M(I)} s_{a,m,I}^2\right)^{1/2} \\
& \geq d \# \Delta_I^{1/2} \left( \frac{a}{2^{4^{I-1}}} \right)^2 \left( \underbrace{\frac{1}{M(i)} \sum_{m=1}^{M(i)} m^4}_{\geq M(I)^4/5} \right)^{1/2} \\
& \geq \frac{d}{\sqrt{5}} \# \Delta_I^{1/2} \left( \frac{a}{2^{4^{I-1}}} \right)^2 M(I)^2.
\end{align*}
Combining this estimate with \eqref{saim}, we obtain
\begin{equation}
\label{eq:est_max_lambdak}
\max_{k \in \Delta_I} \lambda_k = \max_{1 \leq m \leq M(I)} \frac{s_{a,m,I}}{S_{a,I}} \leq d2^d \# \Delta_I^{-1/2}.
\end{equation}

Note that \eqref{sami_discrete} implies
$$
\left| \sum_{m=1}^{M(I)} \sum_{j=0}^{d-1} (\sin (\pi 2^j m x))^2 \sum_{k \in \Delta_I^{(m)}} \cos (2 \pi 2^{k+I m} x) - \sum_{m=1}^{M(I)} s_{a,m,I} \sum_{k \in \Delta_I^{(m)}} \cos (2 \pi 2^{k+I m} x) \right| \leq 2^{3^I + I M(I)} \frac{2^{d+1} \pi m}{2^{4^{I-1}}} \leq 1,
$$
for $x \in \left[\frac{a}{2^{4^{I-1}}},\frac{a+1}{2^{4^{I-1}}} \right]$ and for sufficiently large $I \in\mathbb N$. This means, for the quantity in \eqref{Eq_sin_squared_term}, we have
\begin{align*}
& \mathbb{P} \left[  x \in [0,2^{-h_I}], \, \left | \sum_{m=1}^{M(I)} \sum_{j=0}^{d-1} \left( \sin(2 \pi 2^j m x) \right)^2  \sum_{k \in \Delta_I^{(m)}} \cos(2 \pi 2^{k+I m} x ) \right| > \delta \sqrt{2 N_I \log g_{N_I}}  \right] \\
& = \sum_{a=0}^{2^{4^{I-1}-h_I}-1} \mathbb{P} \left[ x \in \left [ \frac{a}{2^{4^{I-1}}}, \, \frac{a+1}{2^{4^{I-1}}} \right],  \, \left | \sum_{m=1}^{M(I)} \sum_{j=0}^{d-1} \left( \sin(2 \pi 2^j m x) \right)^2  \sum_{k \in \Delta_I^{(m)}} \cos(2 \pi 2^{k+I m} x ) \right| > \delta \sqrt{2 N_I \log g_{N_I}} \right] \\
& \leq
\sum_{a=0}^{2^{4^{I-1}-h_I}-1} 
\mathbb{P} \left[ x \in \left[\frac{a}{2^{4^{I-1}}},\frac{a+1}{2^{4^{I-1}}} \right], \, \left| \sum_{m=1}^{M(I)} \sum_{k \in \Delta_I^{(m)}} s_{a,m,I} \cos (2 \pi 2^{k+I m} x) \right| > \delta \sqrt{ \# \Delta_I \log g_{N_I}} \right], \\
\end{align*}
where we used $N_I = \max \Delta_I \geq \# \Delta_I$ together with the rough estimate $1 \leq (\sqrt{2}-1) \delta \sqrt{ \# \Delta_I \log g_{N_I}}$ in the last step.
\\

Thus, for each $0 \leq a \leq 2^{4^{I-1}-h_I}-1$, using Lemma \ref{lemma_gaposhkin} with the weights $\lambda_k$ as specified above, and using \eqref{eq:est_max_lambdak} as well as the estimate $S_{a,I} \leq  \frac{\# \Delta_I}{10d}$, we have
\begin{eqnarray*}
& & \mathbb{P} \left[ x \in \left[\frac{a}{2^{4^{I-1}}},\frac{a+1}{2^{4^{I-1}}} \right], \, \left| \sum_{m=1}^{M(I)} \sum_{k \in \Delta_I^{(m)}} s_{a,m,I} \cos (2 \pi 2^{k+I m} x) \right| > \delta \sqrt{ \# \Delta_I  \log g_{N_I}} \right] \\
& = &  \mathbb{P} \left[x \in \left[\frac{a}{2^{4^{I-1}}},\frac{a+1}{2^{4^{I-1}}} \right], \, \left|\sum_{m=1}^{M(I)} \sum_{k \in \Delta_I^{(m)}} \lambda_k \cos (2 \pi 2^{k+I m}  x) \right| > S_{a,I}^{-1}  \delta \sqrt{ \# \Delta_I  \log g_{N_I}}  \right] \\
& \leq & \mathbb{P} \left[ x \in \left[\frac{a}{2^{4^{I-1}}},\frac{a+1}{2^{4^{I-1}}} \right], \, \left| \sqrt{2} \sum_{m=1}^{M(I)} \sum_{k \in \Delta_I^{(m)}} \lambda_k \cos (2 \pi 2^{k+Im} x) \right| >   \delta \sqrt{200 d^2  \log g_{N_I}} \right] \\ 
& = & 2^{-4^{I-1}} \left( 1 - \Phi(\delta \sqrt{200 d^2  \log g_{N_I}}) + O \left( \Delta_I^{-1/4} \right) \right) \\
& = &  2^{-4^{I-1}} \Phi \left(- \delta \sqrt{200 d^2 \log g_{N_I}} \right)(1 +o(1)) 
\end{eqnarray*}
uniformly in $a$ for sufficiently large $I \in\mathbb N$. The equality in the last step holds by choosing $d \in \N$ sufficiently large (relative to $\delta > 0$), and hence the tail of $\Phi$ asymptotically dominates $\# \Delta_I^{-1/4}$. By summing over all $0 \leq a \leq 2^{4^{I-1}-h_I-1}$, we arrive at
\begin{equation*}
\begin{split}
& \mathbb{P} \left[ x \in [0,2^{-h_I}], \left | \sum_{m=1}^{M(I)} \sum_{j=0}^{d-1} \left( \sin(2 \pi 2^j m x) \right)^2  \sum_{k \in \Delta_I^{(m)}} \cos(2 \pi 2^{k+I m} x ) \right| > \delta \sqrt{ 2N_I \log g_{N_I}} \right] \\
&=2^{-h_I} \Phi \left(- \delta \sqrt{200 d^2 \log g_{N_I}} \right)(1 +o(1)) .
\end{split}
\end{equation*}
For fixed $\delta > 0$, we can now choose $d \in \N$ sufficiently large such that
\[
\Phi \left(- \delta \sqrt{200 d^2 \log g_{N_I}} \right)(1 +o(1)) = \exp \left( - \frac{2s^2}{d^2} \frac{N_I}{\# \Delta_I} \log g_{N_I} \right)o(1),
\]
where we recall that $s= (1 + \varepsilon) \|f\|_2 + 4\delta $. This means, for the quantity in \eqref{Eq_sin_squared_term}, we have established 
\begin{equation}
\label{Eq_est_sin_squared}
\begin{split}
& \mathbb{P} \left[ x \in [0,2^{-h_I}], \left | \sum_{m=1}^{M(I)} \sum_{j=0}^{d-1} \left( \sin(2 \pi 2^j m x) \right)^2  \sum_{k \in \Delta_I^{(m)}} \cos(2 \pi 2^{k+I m} x ) \right| > \delta \sqrt{2 N_I \log g_{N_I}} \right] \\
& \qquad =  \exp \left( - \frac{2s^2}{d^2} \frac{N_I}{\# \Delta_I} \log g_{N_I} \right)o(1).
\end{split}
\end{equation}
Analogously, we can establish the same estimate for the probability in \eqref{Eq_sin_term}, where we get
\begin{equation}
\begin{split}
\label{Eq_est_sin_simple}
& \mathbb{P} \left[ x \in [0,2^{-h_I}], \left | \sum_{m=1}^{M(I)} \sum_{j=0}^{d-1} \sin(2 \pi 2^j m x)   \sum_{k \in \Delta_I^{(m)}} \sin(2 \pi 2^{k+I m} x) \right| > \delta \sqrt{ 2 N_I \log g_{N_I}} \right] \\
 & \qquad =   \exp \left( - \frac{2s^2}{d^2} \frac{N_I}{\# \Delta_I} \log g_{N_I} \right)o(1),
\end{split}
\end{equation}
provided $d \in \N$ is chosen sufficiently large. Combining the estimate for the main term \eqref{Eq_Est_dominating_term} and the previous estimates \eqref{Eq_est_sin_squared}, \eqref{Eq_est_sin_simple}, we end up with
\begin{equation}
\label{Eq_asymp_Y_I}
\mathbb{P} \left[ Z_I \geq s \sqrt{2 N_I \log g_{N_I}} \right] \geq 2^{-h_I} \exp \left( -\frac{2s^2}{d^2} \frac{N_I}{\# \Delta_I} \log g_{N_I} \right) \left(  1 + o\left( 1\right) \right).
\end{equation}
We recall that $Y_I = Z_I + E_I$, where by \eqref{size_of_E} we now get
\[
\mathbb{P} \left[  Y_I \geq ( (1 + \varepsilon) \| f\|_2 + \delta) \sqrt{ 2 N_I \log g_{N_I}} \right] \geq \exp \left( -\frac{2s^2}{d^2} \frac{N_I}{\# \Delta_I} \log g_{N_I} \right) \left(  1 + o\left( 1\right) \right) - O(N_I^{-1/16}).
\]
Since $\frac{s^2}{2d^2} \leq \frac{2}{d}$ (assuming that $\varepsilon$ and $\delta$ are small) and $\frac{N_I}{\# \Delta_I} \rightarrow 1$, we can again choose $d$ large enough such that
\[
\exp \left( -\frac{2s^2}{d^2} \frac{N_I}{\# \Delta_I} \log g_{N_I} \right) \left(  1 + o\left( 1\right) \right) - O(N_I^{-1/16}) = \exp \left( -\frac{2s^2}{d^2} \log g_{N_I} \right) \left(  1 + o\left( 1\right) \right).
\]

Thus, we arrive at
\begin{align*}
\mathbb{P} \left[ S_{N_I} \geq (1+\varepsilon) \| f\|_2 \sqrt{ 2 N_I \log g_{N_I}} \right] 
& = 
2^{-h_I} \exp \left( -\frac{2s^2}{d^2} \log g_{N_I} \right) \left(  1 +  o\left( 1\right) \right).
\end{align*}
Let $\Phi$ be the standard normal distribution function, then it is well-known that
\[
 1-\Phi((1+ \varepsilon) \sqrt{2\log g_{N_I}}) = \exp(-(1+ \varepsilon)^2 \log g_{N_I})(1+o(1)).
\]
By construction we have $2^{-h_I} \geq 4^{-d} \exp( - \log g_{N_I})$, leading to the estimate
\[
\frac{\mathbb{P} \left[ S_{N_I} \geq (1+\varepsilon) \| f\|_2 \sqrt{ 2 N_I \log g_{N_I}} \right]}{1-\Phi((1+ \varepsilon) \sqrt{2\log g_{N_I}})} \geq 4^{-d}\exp \left( \underbrace{\left ( - \frac{2s^2}{d^2} -1 + (1 + \varepsilon)^2 \right)}_{=: \rho(d, \varepsilon, \delta)} \log g_{N_I} \right).
\]
Since $s=(1+ \varepsilon) \| f \|_2 + 4 \delta$, we can choose $d \in \N$ large enough such that $\rho(d, \varepsilon, \delta) > 0$. The proof is concluded by taking the limit as $I \rightarrow \infty$.
\end{proof}

\section*{Acknowledgement}
CA was supported by the Austrian Science Fund (FWF), projects 10.55776/I4945,
10.55776/I5554, 10.55776/P34763, 10.55776/P35322, and 10.55776/PAT5120424. JP was supported by the DFG project 516672205.

\bibliographystyle{abbrv}
\bibliography{bibliography}

\end{document}